\documentclass{bmcart}

\def\includegraphics{}

\startlocaldefs
\usepackage{amsmath}
\usepackage{graphics}
\usepackage{latexsym}
\usepackage{amsfonts}
\usepackage{mathrsfs}
\usepackage{amssymb}
\usepackage{amsthm}
\usepackage{bbm}
\usepackage{enumerate}
\usepackage{accents}
\usepackage{statex}
\usepackage{xcolor}
\numberwithin{equation}{section}
\RequirePackage[colorlinks,citecolor=blue,urlcolor=blue]{hyperref}

\newtheorem{definition}{Definition}[section]
\newtheorem{theorem}{Theorem}[section]

\newtheorem{lemma}{Lemma}[section]
\newtheorem{remark}{Remark}[section]

\newtheorem{assumption}{Assumption}[section]

\newcommand{\ba}{\begin{array}}
\newcommand{\ea}{\end{array}}
\newcommand{\bal}{\begin{aligned}}
\newcommand{\eal}{\end{aligned}}
\newcommand{\bals}{\begin{aligned*}}
\newcommand{\eals}{\end{aligned*}}
\newcommand{\be}{\begin{equation}}
\newcommand{\ee}{\end{equation}}
\newcommand{\bes}{\begin{equation*}}
\newcommand{\ees}{\end{equation*}}
\newcommand{\bea}{\begin{eqnarray}}
\newcommand{\eea}{\end{eqnarray}}
\newcommand{\beas}{\begin{eqnarray*}}
\newcommand{\eeas}{\end{eqnarray*}}

\newcommand{\abs}[1]{\left| #1\right|}

\def\no{\noindent}

\def\ss{\smallskip}

\def \d {\delta}
\def \R {{\bf R}}

\def \esp {{[0,T]\times \R^m}}

\def \bp {{\bf{P}}}
\def \tx {(t,x)\in \esp}
\def \xs {X_s^{t,x}}
\def \xu {X_u^{t,x}}
\def \xso {X_s^{0,x}}
\def \xto {X_t^{0,x}}

\def \xTo {X_T^{0,x}}

\def \ito {It\^{o}}
\def \E {\textbf{E}}
\endlocaldefs
\begin{document}

\begin{frontmatter}

\begin{fmbox}
\dochead{Research}


\title{One Kind of Multiple Dimensional Markovian BSDEs with Stochastic Linear Growth Generators}


\author[
   addressref={murui},                   
   noteref={m},                        
   email={rui.mu.sdu@gmail.com}   
]
{\inits{mr}\fnm{Rui} \snm{Mu}}
\author[
   addressref={wuzhen},
   corref={wuzhen},
   noteref={w},
   email={wuzhen@sdu.edu.cn}
]{\inits{wz}\fnm{Zhen} \snm{Wu}}


\address[id=murui]{%
  \orgname{School of Mathematics, Shandong University}, 
  \postcode{250100}                                
  \city{Jinan},                              
  \cny{P. R. China};
}
\address[id=wuzhen]{%
  \orgname{School of Mathematics, Shandong University},
  \postcode{250100}
  \city{Jinan},
  \cny{P. R. China}
}


\begin{artnotes}
\note[id=m]{Supported in part by the Natural Science Foundation for Young Scientists of Jiangsu Province, P.R. China (No. BK20140299)                               } 
\note[id=w]{Supported in part by National Natural Science Foundation of China (11221061 and 61174092), 111 project (B12023), the National Natural Science Fund for Distinguished Young Scholars of China (11125102).}
\end{artnotes}

\end{fmbox}


\begin{abstractbox}

\begin{abstract} 
In this article, we deal with a multiple dimensional coupled Markovian BSDEs system with stochastic linear growth generators with respect to volatility processes. An existence result is provided by using approximation techniques.
\end{abstract}


\begin{keyword}
\kwd{Backward Stochastic Differential Equations}
\kwd{Stochastic Differential Equations}
\kwd{Approximation Techniques}
\end{keyword}

\begin{keyword}[class=AMS]
\kwd[39A05]{}
\kwd[; 60H10]{}
\end{keyword}

\end{abstractbox}
%

\end{frontmatter}



\section{Introduction}\label{sec:intro}

Backward stochastic differential equations (BSDEs) was proposed firstly by Bismut (1973) \cite{bismut} in linear case to solve the optimal control problems. Later this notion was generalized by Pardoux and Peng (1990)  \cite{peng1990} into the general nonlinear form and the existence and uniqueness results were proved under the classical Lipschitz condition. A class of BSDE is also introduced by Duffie and Epstein (1992) \cite{de} in point of view of recursive utility in economics. During the past twenty years, BSDEs theory attracts many researchers' interests and has been fully developed into various directions. Among the abundant literature, we refer readers the florilegium book edited by El-Karoui and Mazliark (1997) \cite{el1997backward} for the early works before 1996. Surveys on BSDEs theory also includes \cite{el2008backward} which is written by El-Karoui, Hamad{\`e}ne and Matoussi collected in book (2009) \cite{carmona2009indifference}  (see Chapter 8) and the book by Yong and Zhou (1999) \cite{yong1999stochastic} (see Chapter 7).  Some applications on optimization problems can be found in \cite{el2008backward}. About Other applications such as in field of economics, we refer to El-Kaoui, Peng and Quenez (1997) \cite{karoui} . Recently, a complete review on BSDEs theory as well as some new results on nonlinear expectation are introduced in a survey paper by Peng (2011) \cite{peng2011}. 
  
One possible extension to the pioneer work of \cite{peng1990} is to relax as much as possible the uniform Lipschitz condition on the coefficient. A weaker hypothesis is presented by Mao (1995) \cite{mao} which we translate  as follows: for all $y,\bar y, z, \bar z$ and $t\in [0,T]$, the generator of the BSDE satisfies $|f(t,y,z)-f(t,\bar y, \bar z)|^2\leq \kappa |y-\bar y|^2+c|y-\bar y|^2$ a.s. where $c>0$ and $\kappa$ is a concave non-decreasing function from $\R^+$ to $\R^+$ such that $\kappa(0)=0,\ \kappa(u)>0$ for $u>0$ and $\int_{0^+}1/\kappa(u)du=\infty$. An  existence and uniqueness result is proved under such condition in \cite{mao}. Hamad\`ene introduced in (1996) \cite{hamadene1996equations}   a one-dimensional BSDE with local Lipschitz generator. Later Lepeltier and San Martin (1997) \cite{lepeltier1997backward} provided an existence result of minimum solution for one dimensional BSDE where the generator function $f$ is continuous and of linear growth in terms of $(y, z)$. When $f$ is uniformly continuous in $z$ with respect to $(\omega,t)$ and independent of $y$, a uniqueness result was obtained by Jia \cite{jia2008uniqueness}. BSDEs with polynomial growth generator is studied by Briand in \cite{briand2000bsdes}. The case of 1-dimensional BSDEs with coefficient which is monotonic in $y$ and non-Lipschitz on $z$ is shown in work \cite{briand2007one}. About the BSDE with continuous and quadratic growth driver, a classical research should be the one by Kobylanski \cite{kobylanski} (2000) which investigated a one-dimensional BSDE with driver $|f(t,y,z)\leq C(1+|y|+|z|^2)$ and bounded terminal value. This result was generated by Briand and Hu into the unbounded terminal value case in \cite{briand2006bsde} (2006). 

There are plenty works on one-dimensional BSDE. However, limited results have been obtained about the multi-dimensional case. We refer Hamad\`ene, Lepeltier and Peng \cite{hamadene1997} for an existence result on BSDEs system of Markovian case where the driver is of linear growth on $(y,z)$ and of polynomial growth on the state process. See Bahlali \cite{bahlali2001backward} \cite{bahlali2002existence} for high-dimension BSDE with local Lipschitz coefficient. 

In the present article, we consider a high dimensional BSDE under Markovian framework as follows:
\[
Y_t^i=g^i(X_T)+\int_t^TH_i(s, X_s, Y_s^1,...,Y_s^n,Z_s^1,...,Z_s^n)ds-\int_t^T Z_s^idB_s.
\]
for $i=1,2,...,n$, with process $X$ as a solution of a stochastic differential equation (SDE for short). For each $i=1,2,...,n$, the coefficient $H_i$ is continuous on $(y^1,...,y^n,z^1,...,z^n)$ and satisfies:
\[
|H_i(t,x,y^1,...,y^n,z^1,...,z^n)|\leq C(1+|x|)|z^i|+ C(1+|x|^{\gamma}+|y^i|),\quad \gamma>0.
\]
As the generator of the BSDE, actually, $H_i$ is of stochastic linear growth on $Z^i$, or in another word, it is of linear growth $\omega$ by $\omega$. Similar situation was considered in \cite{mu1} in the background of nonzero-sum stochastic differential game problem. However, in \cite{mu1}, the generator $H_i$ is independent on $(y^1,...,y^n)$. According to our knowledge, this general form of high dimensional coupled BSDEs system with stochastic linear growth generator has not been considered in literature. This is the main motivation of the present work.

The rest of this article is organized as follows: in Section 2, we give some notations ans assumptions on the coefficient. The properties of the forward SDE are also provided. The main existence result of BSDEs is proved in Section 3 where a measure domination result plays an important role. This domination result holds true when we assume that the diffusion process of the SDE satisfies the uniform elliptic condition.  For the proof of the main result, we adopt an approximation scheme following the well know mollify technique. The irregular coefficients are approximated by a sequence of Lipschiz functions. Then, we obtain the uniform estimates of the sequence of solutions as well as the convergence result in some appropriate spaces. Finally, we verify that the limit of the solutions is exactly the solution to the original BSDE which completes the proof.
\section{Notations and assumptions}\label{sec:state}

In this section, we will give some basic notations, the preliminary assumptions throughout this paper, as well as some useful results. Let $(\Omega, \mathcal{F}, \bf{P})$ be a probability space on which we
define a $m$-dimensional Brownian motion $B=(B_t)_{0\leq t\leq T}$ with integer $m\geq 1$. Let us denote by ${\bf{F}}=\{\mathcal{F}_t, 0\leq t\leq T\}$ for fixed $T>0$, the natural filtration generated by process $B$ and augmented by
$\mathcal{N}_{\bf{P}}$ the $\bf{P}$-null sets, \emph{i.e.}
$\mathcal{F}_{t}=\sigma\{B_{s},$\ $s\leq t\}\vee\mathcal{N}_{\bf{P}}%
$.

Let $\mathcal{P}$ be the $\sigma$-algebra on $[0,T]\times \Omega$ of
$\mathcal{F}_t$-progressively measurable sets. Let $p\in [1,\infty)$ be real constant and $t\in [0,T]$ be fixed, We then define the following spaces: 
$\mathcal{L}^p$ = $\{\xi: \mathcal{F}_t$-measurable and$\ \R^m$-valued random variable s.t. $\ \textbf{E}[|\xi|^p]<\infty\}$; 
 $\mathcal{S}_{t,T}^p$ = $\{\varphi=(\varphi_s)_{t\leq s \leq T}$: $\mathcal{P}$-measurable and $\R^m$-valued s.t. $\textbf{E}[\sup_{s\in [t,T]}|\varphi_s|^p]< \infty \}$ and 
 $\mathcal{H}_{t,T}^p=\{\varphi=(\varphi_s)_{t \leq s\leq T}: \mathcal{P}$-measurable and $\R^m$-valued s.t.  $\textbf{E}[(\int_{t}^T|\varphi_s|^2ds)^{\frac{p}{2}}]< \infty\}$. 
 Hereafter, $\mathcal{S}_{0,T}^p$ and $\mathcal{H}_{0,T}^p$ are simply denoted by $\mathcal{S}_T^p$ and $\mathcal{H}_{T}^p$. 
 
The following assumptions are in force throughout this paper. Let
$\sigma$ be the function defined as:
\[
\sigma: [0,T]\times \R^m \longrightarrow \R^{m\times
m}
\]
which satisfies the
following assumption.
\begin{assumption}\label{as:sigma}
   \begin{description}
      \item[(i)] $\sigma$ is uniformly Lipschitz w.r.t $x$. \emph{i.e.} there exists a constant $C_1$  such that,
  $\forall t \in [0, T],\  \forall x, x^{\prime} \in \R^m,\quad
| \sigma(t,x)-\sigma(t, x^{\prime})| \leq C_1 |x-
x^{\prime}|.$
      \item[(ii)] $\sigma$ is invertible and bounded and its inverse is bounded, \textit{i.e.}, there exits a constant $C_{\sigma}$ such that $\forall (t,x)\in \esp, \quad \abs {\sigma(t, x)} +\abs{\sigma^{-1}(t,x)}\leq C_{\sigma}.$
   \end{description}
\end{assumption}

\begin{remark} {\bf Uniform elliptic  condition.}\\
\noindent Under Assumptions \ref{as:sigma}, we can verify that, there exists a real constant $\epsilon>0$ such that for any $(t,x)\in \esp $,
\begin{equation}\label{eq:horm}
\epsilon.I\leq \sigma(t,x).\sigma^\top(t,x)\leq \epsilon^{-1}.I
\end{equation}
where $I$ is the identity matrix of dimension $m$. 
\end{remark}

Suppose that we have a system whose dynamic is described by a stochastic differential equation as follows: for $\tx$,
\be\label{eq:SDE sigma} 
\left\{ \bal
\xs &= x+ \int_t^s \sigma(u,\xu)dB_u%
,\ s\in\left[t,T\right]  ;\ss\\
\xs &= x,\ s\in [0,t]. \eal \right.
\ee

The solution $X= (\xs)_{s\leq T}$ exists and is unique under Assumption \ref{as:sigma}. (cf. Karatzas and Shreve 1991 \cite{KS}, p.289). We recall a well-known result associates to integrability of the solution. For any fixed $(t, x)\in [0,T]\times \R^m$, $p\geq 2$, it holds that, $\bp$-a.s.,
\begin{equation}\label{eq:est_x}
\E\left[\sup_{0\leq s\leq T}|X_s^{t, x}|^p\right]\leq C(1+|x|^p),
\end{equation}
where the constant $C$ is only depend on the Lipschitz coefficient and the bound of $\sigma$.  In addition, property \eqref{eq:est_x} holds true, as well, for the expectation under the probability which is equivalent to $\bp$.

For integer $n\geq 1$, we first present the following Borelian function as the terminal coefficient of the $n$-dimensional BSDE that we considering:
$$g^i:\R^m\longrightarrow \R, \ i=1,2,...,n.
$$
which satisfy
\begin{assumption}\label{as:hg}
      The function $g^i$, $i=1,2,...,n$, are of polynomial growth with respect to $x$, \emph{i.e.} %
  there exist constants $C_g$ and $\gamma\geq 0$ such that 
  $$ |g^i(x)|\leq C_g(1+|x|^{\gamma}),
\forall x\in \R^m , \text{ for } i=1,2,...,n.$$
\end{assumption}

Now, we consider Borelian functions $H_i, i=1,2,...,n$, from $[0,T]\times \R^m\times \R^n\times \R^{nm}$ into $\R$ as follows:
\[
H_i(t, x, y^1,...,y^n,z^1,...,z^n),\quad i=1,2,...,n
\]
which satisfy the following hypothesis:

\begin{assumption}\label{as:H}
\begin{description}
\item [(i)] For each $(t,x,y^1,...,y^n,z^1,...,z^n)\in [0,T]\times \R^m\times \R^n\times \R^{nm}$, there exists constants $C_2$, $C_h$ and $\gamma>0$, such that, for each $i=1,2...,n$,
\begin{equation}\label{eq:H}
|H_i(t,x,y^1,...,y^n,z^1,...,z^n)|\leq C_2(1+|x|)|z^i|+ C_h(1+|x|^{\gamma}+|y^i|) ;
\end{equation}

\item[(ii)] the mapping $(y^1,..., y^n, z^1,...,z^n)\in \R^n\times  \R^{nm}\longmapsto H_i(t,x,y^1,..., y^n,z^1,...,z^n) \in \R$ is continuous for any fixed $(t,x) \in [0,T]\times \R^m$.
\end{description}
\end{assumption}

The BSDE that we concern in this work is the following:
\begin{equation}\label{eq:mainbsde}
Y_t^i=g^i(\xTo)+\int_t^TH_i(s, \xso, Y_s^1,...,Y_s^n,Z_s^1,...,Z_s^n)ds-\int_t^T Z_s^idB_s.
\end{equation}
for $i=1,2,...,n$. From Assumptions \ref{as:hg} and \ref{as:H}, we know that this is a multiple dimensional coupled BSDEs system under Markovian framework, with unbounded terminal value and the generator of the $i_{th}$ equation is of linear growth on $Z^i$ component $\omega$ by $\omega$, or in another word, it is of stochastic linear growth.

\section{Existence of solutions for the multiple dimensional coupled BSDEs system}\label{sec:couple}

In this section, we will provide an existence result of BSDEs \eqref{eq:mainbsde} when $n=2$ as an example. Actually, the case for $n>2$ can be dealt with in the same way without any difficulties. 

\subsection{Measure domination}

Before we state our main theorem, let us first recall a result related to measure domination.
\begin{definition}\label{cdtdom}\textbf{: $L^{q}$-Domination condition}
\\
Let $q\in (1,\infty)$ be fixed. For a given $t_1\in [0,T]$, a family of probability measures
$\{\nu_1(s,dx), s\in [t_1, T]\}$ defined on $\R^m$ is said to be $L^{q}$-
dominated by another family of probability measures
$\{\nu_0(s,dx), s\in [t_1, T]\}$, if for any $\delta \in(0, T-t_1]$,
there exists an application $\phi^\d_{t_1}: [t_1+\d, T]\times \R^m \rightarrow
\R^+$ such that:
\\

(i) $\nu_1$(s, dx)ds= $\phi^\d_{t_1}$(s, x)$\nu_0$(s, dx)ds on $[t_1+\delta, T]\times$
  $\R^m$.
  
(ii) $\forall k\geq 1$, $\phi^\d_{t_1}(s,x) \in L^q([t_1+\delta, T]\times [-k, k]^m$; $\nu_0(s,
  dx)ds)$.
\end{definition}

\begin{lemma}\label{lem:dom}
Let $x_0\in \R^m$, $\tx$, $s\in (t,T]$ and 
$\mu(t,x;s,dy)$ the law of $X^{t,x}_s$, \textit{i.e.}, 
$$\forall A \in \mathcal{B}(\R^m),\,\,
\mu(t,x;s,A)= \bp(X_s^{t,x}\in A).$$ Under Assumption \ref{as:sigma} on $\sigma$, for any $q\in (1,\infty)$, the family 
of laws $\{\mu(t,x;s,dy), s\in [t,T]\}$ is $L^q$-dominated by 
$\{\mu(0,x_0;s,dy), s\in [t,T]\}$ for fixed $x_0\in \R^m$.
\end{lemma} 

\begin{proof}
See \cite{mu1}, Lemma 4.3 and Corollary 4.4, pp. 14-15.
\end{proof}

\subsection{High dimensional coupled BSDEs system}
Our main result in this section is the following theorem.
\begin{theorem}\label{th:main}
Let $x\in \R^m$ be fixed. Then under Assumptions \ref{as:sigma},  \ref{as:hg} and \ref{as:H}, there exist two pairs of $\mathcal{P}$-measurable processes $(Y^i, Z^i)$ with values in $\R^{1+m}$,  $i=1,2,$ and two deterministic functions $\varsigma^i(t,x)$ which are of polynomial growth, i.e. $|\varsigma^i(t,x)|\leq C(1+|x|^{\gamma})$ with $\gamma\geq 0$, $i=1,2$ such that, 
\begin{equation}\label{eq:mainbsdeinth}
\left\{
\begin{aligned}
\bp&\text{-a.s.}, \forall t\leq T, Y_t^i=\varsigma^i(t,X_t^{0,x}) \text{ and } Z^i \text{is dt-square integrable } \bp\text{-a.s.};\\  
Y_t^1&=g^1(\xto)+\int_t^T H_1(s, \xso, Y_s^1,Y_s^2,Z_s^1,Z_s^2) ds-\int_t^T Z_s^1dB_s;\\
Y_t^2&=g^2(\xto)+\int_t^T H_2(s, \xso, Y_s^1,Y_s^2,Z_s^1,Z_s^2) ds-\int_t^T Z_s^2dB_s.\\
\end{aligned} 
\right.
\end{equation} 
The result holds true as well for case $i=n>2$ following the same way.
\end{theorem}
\begin{proof}

The structure of this proof is as follows. We first use the mollify technique on the generator $H_i$, to construct a sequence of BSDEs with generators which are of  Lipschitz continuous. Then, we provide uniform estimates of the solutions, as well as the convergence property. Finally, we verify that the limits of the sequences are exactly the solutions for BSDE \eqref{eq:mainbsdeinth}.

\paragraph{Step 1.}\label{step1} Approximation.

Let $\xi$ be an element of $C^{\infty}(\R^{2+2m}, \R)$ with compact support and satisfies:
\[
\int_{\R^{2+2m}}\xi(y^1,y^2,z^1,z^2)dy^1dy^2dz^1dz^2=1.
\]
\noindent For $(t,x,y^1,y^2,z^1,z^2)\in [0,T]\times \R^m\times \R^{2+2m}$, we set,
\begin{align*}
\tilde{H}_{1n}(&t,x,y^1,y^2,z^1,z^2)\\
&=\int_{\R^{2+2m}}n^4 H_1(s, \varphi_n(x),p^1,p^2,q^1,q^2)\cdot \\
&\quad \xi\left(n(y^1-p^1),n(y^2-p^2),n(z^1-q^1),n(z^2-q^2)\right)dp^1dp^2dq^1dq^2,
\end{align*}
where the truncation function $\varphi_n(x)=((x_j\vee (-n))\wedge n)_{j=1,2,...,m}$, for $x=(x_j)_{j=1,2,...,m}\in \R^m$.
\no We next define $\psi\in C^{\infty}(\R^{2+2m},\R)$ by,
\begin{equation*}
\psi(y^1,y^2,z^1,z^2)=
\left\{
\begin{aligned}
1,\quad |y^1|^2+|y^2|^2+|z^1|^2+|z^2|^2\leq 1,\\
0,\quad |y^1|^2+|y^2|^2+|z^1|^2+|z^2|^2\geq 4.
\end{aligned}
\right.
\end{equation*}
\no Then, we define the measurable function sequence $(H_{1n})_{n\geq 1}$ as follows:\\
 $\forall (t,x,y^1,y^2,z^1,z^2)\in [0,T]\times \R^m\times\R^{2+2m}$,
\[
H_{1n}(t,x,y^1,y^2,z^1,z^2)=\psi(\frac{y^1}{n},\frac{y^2}{n},\frac{z^1}{n},\frac{z^2}{n})\tilde{H}_{1n}(t,x,y^1,y^2,z^1,z^2)
\]
We have the following properties:
\begin{equation}\label{abcd}
\left\{
\begin{array}[c]{l}%
\ (a)\  H_{1n}\  \text{is uniformly lipschitz w.r.t}\ (y^1,y^2,z^1,z^2);\smallskip\\
\ (b)\  |H_{1n}(t,x,y^1,y^2,z^1,z^2)|\leq C_2(1+|\varphi_n(x)|)|z^1|+C_h(1+|\varphi_n(x)|^{\gamma}+|y^1|); \smallskip\\
\ (c)\ |H_{1n}(t,x,y^1,y^2,z^1,z^2)|\leq c_n, \ \text{for any}\  (t,x,y^1,y^2,z^1,z^2);\smallskip\\
\ (d)\ \text{For any}\  (t,x)\in [0,T]\times \R^m,\text{ and }\textbf{K}\ \text{a compact subset of}\  \R^{2+2m},\\
\qquad \sup\limits_{(y^1,y^2,z^1,z^2)\in \textbf{K}}\abs{H_{1n}(t,x,y^1,y^2,z^1,z^2)-H_1(t,x,y^1,y^2,z^1,z^2)}\rightarrow 0,\\
 \qquad \text{as} \ n\rightarrow \infty.
\end{array}
\right.
\end{equation}

The construction of the approximating sequence $(H_2^n)_{n\geq 1}$ is carried out in the same way.

For each $n\geq 1$ and $(t,x)\in [0,T]\times \R^m$, since $H_{1n}$ and $H_{2n}$ are uniformly lipschitz w.r.t $(y^1,y^2,z^1,z^2)$, by the result of Pardoux-Peng  (see \cite{peng1990}), we know, there exist two pairs of processes $(\bar{Y}^{in;(t,x)}, \bar{Z}^{in;(t,x)})\in \mathcal{S}_{t,T}^2(\R)\times \mathcal{H}_{t,T}^{2}(\R^m)$, $i=1,2$, which satisfy, for $s\in[t,T]$,
\begin{equation}\label{eq:BSDEseq}
\left\{
\begin{aligned}
Y_s^{1n;(t,x)}&=g^1(X_T^{t,x})+\int_s^T H_{1n}(r,X_r^{t,x}, Y_r^{1n;(t,x)},Y_r^{2n;(t,x)},Z_r^{1n;(t,x)},Z_r^{2n;(t,x)})dr\\
&\qquad \qquad \quad-\int_s^T Z_r^{1n;(t,x)}dB_r;\smallskip\\
Y_s^{2n;(t,x)}&=g^2(X_T^{t,x})+\int_s^T H_{2n}(r,X_r^{t,x}, Y_r^{1n;(t,x)},Y_r^{2n;(t,x)},Z_r^{1n;(t,x)},Z_r^{2n;(t,x)})dr\\
&\qquad \qquad \quad-\int_s^T Z_r^{2n;(t,x)}dB_r.
\end{aligned}
\right.
\end{equation}
\no Meanwhile, the properties \eqref{abcd}-(a),(c) and the result of El-karoui et al. (ref. \cite{karoui}) yield that, there exist two sequences of deterministic measurable applications $\varsigma^{1n}(resp.\ \varsigma^{2n}): [0,T]\times \R^m \rightarrow \R$ and $\mathfrak{z}^{1n}(resp.\  \mathfrak{z}^{2n}): [0,T]\times \R^m\rightarrow \R^m$ such that for any $s\in [t,T]$,
\begin{equation}\label{eq:varsigma}
Y_s^{1n;(t,x)}= \varsigma^{1n}(s,X_s^{t,x})\quad (resp.\  Y_s^{2n;(t,x)}=\varsigma^{2n}(s,X_s^{t,x}))
\end{equation}
and
\[
Z_s^{1n;(t,x)}= \mathfrak{z}^{1n}(s,X_s^{t,x}) \quad (resp.\  Z_s^{2n;(t,x)}=\mathfrak{z}^{2n}(s,X_s^{t,x})).
\]
Besides, we have the following deterministic expression: for $i=1,2,$ and $n\geq 1$,
\begin{equation}\label{eq:uin}
\varsigma^{in}(t,x)= \E\Big[g^i(X_T^{t,x})+\int_t^T F_{in}(s,X_s^{t,x})ds\Big],\quad \forall(t,x)\in [0, T]\times \R^m,
\end{equation}
\no where,
\[
F_{in}(s,x)= H_{in}(s,x,\varsigma^{1n}(s,x),\varsigma^{2n}(s,x),\mathfrak{z}^{1n}(s,x),\mathfrak{z}^{2n}(s,x)).
\]

\paragraph{Step 2.}\label{step2} Uniform integrability of $(Y^{1n;(t,x)}, Z^{1n;(t,x)})_{n\geq 1}$ for fixed $(t,x)\in [0,T]\times \R^m$.

For each $n\geq 1$, let us first consider the following BSDE:
\begin{align}\label{eq:BSDEbar}
\bar {Y}_s^{1n}&=g^1(X_T^{t,x})+\int_t^T C_2(1+|\varphi_n (X_r^{t,x})|)|\bar Z_r^{1n}| \nonumber\\
\qquad \qquad &+C_h(1+|\varphi_n(X_r^{t,x})|^\gamma +
|\bar Y_r^{1n}|)dr -\int_s^T \bar Z_r^{1n} dB_r.
\end{align}
For any $x\in \R^m$ and $n\geq 1$, the mapping $(y^1,z^1)\in \R\times \R^m \mapsto C_2(1+|\varphi_n (X_r^{t,x})|)|z^1| +C_h(1+|\varphi_n(X_r^{t,x})|^\gamma + |y^1|)$ is Lipschitz continuous, therefore, the solution $(\bar Y^{1n}, \bar Z^{1n}) \in \mathcal{S}_{t,T}^2(\R)\times \mathcal{H}_{t,T}^2(\R^m)$ exists and is unique. Moreover, it follows from the result of El-Karoui et al. (see \cite{karoui}) that, $\bar Y^{1n}$ can be characterized through a deterministic measurable function $\bar \varsigma^{1n}: [0,T]\times \R^m \rightarrow \R$, that is, for any $s\in [t,T]$,
\begin{equation}\label{eq:varsigmabar}
\bar {Y}_s^{1n}=\bar \varsigma^{1n}(s, X_s^{t,x}).
\end{equation}
Next let us consider the process 
\[
B^n_s= B_s- \int_0^s 1_{[t,T]}(r)C_2(1+|\varphi_n(X_r^{t,x})|)\mbox{sign}(\tilde{Z}_r^{1n})dr,\,\, 0\leq s\leq T,
\]
which is, thanks to Girsanov's Theorem, a Brownian motion under the probability $\bp^n$ on $(\Omega, \mathcal{F})$ whose density with respect to $\bp$ is
$\mathcal{E}_T:=\mathcal{E}_T(\int_0^T C_2(1+|\varphi_n(X_s^{t,x})|)\mbox{sign}(\bar{Z}_s^{1n})1_{[t,T]}(s) dB_s),$
where for any $z=(z^i)_{i=1,...,d}\in \R^m$, $\mbox{sign}(z)=(1_{[|z^i|\neq 0]}\frac{z^i}{|z^i|})_{i=1,...,d}$ and  $\mathcal{E}_t(\cdot)$ is
defined by 
\begin{equation}\label{eq:density fun}
\mathcal{E}(M):= (\exp\{M_t- \langle M \rangle_t/2\})_{t\leq T},
\end{equation}
for any $(\mathcal{F}_t, \bp)$-continuous local martingale $M=(M_t)_{t\leq T}$. Then \eqref{eq:BSDEbar} becomes 
\begin{equation}
\bar {Y}_s^{1n}=g^1(X_T^{t,x})+\int_t^T C_h(1+|\varphi_n(X_r^{t,x})|^\gamma + |\bar Y_r^{1n}|)dr -\int_s^T \bar Z_r^{1n} dB_r^n, \quad s\in [t,T].
\end{equation}
Then, by using \ito's formula to $e^{C_h t}\bar Y_t^{1n}$, $t\leq T$, we know, 
\begin{align*}
e^{C_h t}\bar Y_t^{1n}= e^{C_h T} g^1(X_T^{t,x})-\int_t^T C_he^{C_h r}(1+|\varphi_n(X_r^{t,x})|^\gamma)dr+\int_t^T e^{C_h r} \bar Z_r^{1n}dB_r^n.
\end{align*}
Considering \eqref{eq:varsigmabar}, we have, 
\begin{equation*}
\bar \varsigma^{1n}(t,x)=e^{-C_h t}\E^n\left[e^{C_h T}g^1(X_T^{t,x})+\int_t^T C_he^{C_h s}(1+|\varphi_n(X_s^{t,x})|^\gamma)ds\Big|\mathcal{F}_t\right],
\end{equation*}
where $\E^n$ is the expectation under the probability $\bp^n$. Taking the expectation on both sides under the probability $\bp^n$ and taking account of $\bar \varsigma^{1n}$ is deterministic, we obtain,
\[
\bar \varsigma^{1n} (t,x)= e^{-C_h t}\E^n\left[e^{C_h T}g^1(X_T^{t,x})+\int_t^T C_he^{C_h s}(1+|\varphi_n(X_s^{t,x})|^\gamma)ds\right].
\]
Since, $g^1$ is of polynomial growth, and $e^{-C_ht}\leq 1$ for $t\in [0,T]$, we infer that, $\forall (t,x)\in [0,T]\times\R^m$,
\begin{align*}
|\bar \varsigma^{1n}(t,x)|
&\leq C \E^n\left[\sup_{t\leq s\leq T} \left(1+|X_s^{t,x}|^\gamma \right) \right]\\
&\leq C \E\left[\mathcal{E}_T \cdot  \sup_{t\leq s\leq T} \left(1+|X_s^{t,x}|^\gamma \right)\right],
\end{align*}
where the constant $C$ depends only on $T,\  C_h$ and $C_g$. It follows from a result of Haussmann (\cite{Haussmann1986}, p.14, see also \cite{mu1}, Lemma 3.1) that, there exists some $p_0\in (1,2)$ and a constant $C$ which is independent on $n$, such that, $\E[|\mathcal{E}_T|^{p_0}]\leq C$ uniformly. As a result of Young's inequality and estimate \eqref{eq:est_x}, we have,
\begin{equation*}
|\bar \varsigma^{1n}(t,x)|
\leq C_{p_0}\{ \E\left[ |\mathcal{E}_T|^{p_0} \right] + \E\left[  \sup_{t\leq s\leq T} \left(1+|X_s^{t,x}|^{\frac{p_0\gamma}{p_0-1}} \right)\right]\},
\end{equation*}
which yields,
\begin{equation}
|\bar \varsigma^{1n}(t,x)| \leq C(1+|x|^\lambda) \text{ with } \lambda= p_0 \gamma/(p_0-1)>2.
\end{equation}
Next by the comparison theorem of BSDEs and property \eqref{abcd}-(b), we actually have, for any $s\in [t,T]$,
\[
Y_s^{1n;(t,x)}=\varsigma^{1n}(s, X_s^{t,x})\leq \bar Y_s^{1n}= \bar{\varsigma}^{1n}(s, X_s^{t,x}),
\]
and by choosing $s=t$, we get that $\varsigma^{1n}(t,x)\leq C(1+|x|^\lambda)$, $(t,x)\in [0,T]\times \R^m$. In a similar way, we can also show that for any $(t,x)\in [0,T]\times \R^m$, $\varsigma^{1n}(t,x)\geq -C(1+|x|^\lambda)$. As a conclusion, $\varsigma^{1n}$ is of polynomial growth on $(t,x)$ uniformly in $n$, i.e. there exists constants $C$ which is independent of $n$, and $\lambda>2$ such that,
\begin{equation}\label{eq:varsigmagrowth}
|\varsigma^{1n}(t,x)|\leq C(1+|x|^\lambda).
\end{equation}
Combining \eqref{eq:varsigmagrowth} and \eqref{eq:est_x}, we deduce that, for any $\alpha>1$, $i=1,2$, 
\begin{equation}\label{eq:estY1n}
\E\left[\sup_{t\leq s\leq T} |Y_s^{in;(t,x)}|^\alpha\right]\leq C.
\end{equation}
On the other hand, by applying \ito's formula with $(Y^{in;(t,x)})^2$ and considering the uniform estimate \eqref{eq:estY1n}, we can infer in a regular way that, for any $t\in [0,T]$, $i=1,2$,
\begin{equation}\label{eq:estZ1n}
\E\left[\int_t^T |Z_s^{in;(t,x)}|^2ds\right]\leq C.
\end{equation}
\paragraph{Step 3.} There exists a subsequence of $((Y_s^{1n;(0,x)},Z_s^{1n;(0,x)})_{0\leq s\leq T})_{n\geq 1}$ which converges respectively to $(Y^1_s, Z^1_s)_{0\leq s\leq T}$, solution of the BSDE \eqref{eq:mainbsdeinth}. 
Let us recall the expression \eqref{eq:uin} for case $i=1$, and apply property \eqref{abcd}-(b) combined with the uniform estimates \eqref{eq:estY1n}, \eqref{eq:estZ1n}, \eqref{eq:est_x} and the Young's inequality to show that, for $1<q<2$,
\begin{equation}\label{eq:F1nlp}
\begin{aligned}
\textbf{E}&\bigg[\int_{0}^T|F_{1n}(s,\xso)|^qds\bigg]\\
&\leq C\textbf{E}\left[\int_{0}^T(1+|\varphi_n(\xso)|)^q|\bar{Z}_s^{1n;(0,x)}|^q+(1+|\varphi_n(\xso)|^{\gamma q}+|\bar{Y}_s^{1n;(0,x)}|^q )ds\right]\\
&\leq C\{\textbf{E}\left[\int_{0}^T |\bar{Z}_s^{1n;(0,x)}|^2ds\right]+ \textbf{E}\left[ \sup_{0\leq s\leq T}|\bar{Y}_s^{1n;(0,x)} |^q\right]+1\}\\
&\leq C.
\end{aligned}
\end{equation}
Therefore, there exists a subsequence $\{n_k\}$ (for notation simplification, we still denote it by $\{n\}$) and a $\mathcal{B}([0, T])\otimes \mathcal{B}(\R^m)$-measurable deterministic function $F_1(s,y)$ such that:
\begin{equation}\label{eq:Fn weakly cov}
F_{1n}\rightarrow F_1 \ \text{weakly in}\ \mathcal{L}^q([0,T]\times \R^m; \mu(0,x;s,dy)ds).
\end{equation}
Next we aim to prove that $(\varsigma^{1n}(t,x))_{n\geq 1}$ is a Cauchy
sequence for each $(t,x)\in[0,T]\times \R^m$. Now let $(t,x)$ be fixed, $\eta>0$, $k$, $n$ and $m\geq1$ be integers. From \eqref{eq:uin}, we have,
\begin{equation*}
\begin{aligned}
|\varsigma^{1n}(t,x)-\varsigma^{1m}(t,x)|&=
\Big|\textbf{E}\Big[\int_t^TF_{1n}(s,X_s^{t,x})-F_{1m}(s,X_s^{t,x})ds\Big]\Big|\\
&\leq
\textbf{E}\Big[\int_t^{t+\eta}\left|F_{1n}(s,X_s^{t,x})-F_{1m}(s,X_s^{t,x})\right|ds\Big]\\
&\quad +\Big|\textbf{E}\Big[\int_{t+\eta}^T\left(F_{1n}(s,X_s^{t,x})-F_{1m}(s,X_s^{t,x})\right).1_{\{X_s^{t,x}|\leq
k\}}ds\Big]\Big|\\
&\quad +\Big|\textbf{E}\Big[\int_{t+\eta}^T\left(F_{1n}(s,X_s^{t,x})-F_{1m}(s,X_s^{t,x})\right).1_{\{|X_s^{t,x}|>k\}}ds\Big]\Big|,
\end{aligned}
\end{equation*}
where on the right side, noticing \eqref{eq:F1nlp}, we obtain,
\begin{equation*}
\begin{aligned}
\textbf{E}&\Big[\int_t^{t+\eta}|F_{1n}(s, X_s^{t,x})- F_{1m}(s,
X_s^{t,x})|ds\Big]\\
& \leq
\eta^{\frac{q-1}{q}}\{\textbf{E}\Big[\int_0^T|F_{1n}(s,X_s^{t,x})-F_{1m}(s,X_s^{t,x})|^qds\Big]\}^{\frac{1}{q}}\leq C\eta^{\frac{q-1}{q}}.
\end{aligned}
\end{equation*}
At the same time, Lemma \ref{lem:dom} associates with the $\mathcal{L}^{\frac{q}{q-1}}$-domination property implies:
\begin{equation*}
\begin{aligned}
\Big|\textbf{E}\Big[\int_{t+\eta}^T&\left(F_{1n}(s,X_s^{t,x})-F_{1m}(s,X_s^{t,x})\right).1_{\{\abs{X_s^{t,x}}\leq
k\}}ds\Big]\Big|\\
&=
\Big|\int_{\R^m}\int_{t+\eta}^T(F_{1n}(s,\eta)-F_{1m}(s,\eta)).1_{\{\abs{\eta}\leq
k\}}\mu(t,x;s,d\eta)ds\Big|\\
&=
\Big|\int_{\R^m}\int_{t+\eta}^T(F_{1n}(s,\eta)-F_{1m}(s,\eta)).1_{\{\abs{\eta}\leq
k\}}\phi_{t,x}(s,\eta)\mu(0,x;s,d\eta)ds\Big|.
\end{aligned}
\end{equation*}
Since $\phi_{t,x}(s,\eta) \in \mathcal{L}^{\frac{q}{q-1}}([t+\eta,
T]\times [-k, k]^m$; $\mu(0,x; s, d\eta)ds)$, for $k\geq 1$, it
follows from \eqref{eq:Fn weakly cov} that for each $(t,x)\in [0, T]\times \R^m$, we have,
\begin{equation*}
\textbf{E}\Big[\int_{t+\eta}^T\left(F_{1n}(s, X_s^{t,x})-F_{1m}(s,
X_s^{t,x})\right)\mathbbm{1}_{\{\abs{X_s^{t,x}}\leq
k\}}ds\Big]\rightarrow 0\text{ as } n,m\rightarrow \infty.
\end{equation*}
\noindent Finally,
\begin{equation*}
\begin{aligned}
\Big|\textbf{E}&\Big[\int_{t+\eta}^T\left(F_{1n}(s,X_s^{t,x})-F_{1m}(s,X_s^{t,x})\right).\mathbbm{1}_{\{\abs{X_s^{t,x}}>
k\}}ds\Big]\Big|\\
&\leq C\{\textbf{E}\Big[\int_{t+\eta}^T
\mathbbm{1}_{\left\{\abs{X_s^{t,x}}>k\right\}}ds\Big]\}^{\frac{q-1}{q}}\{\textbf{E}\Big[\int_{t+\eta}^T\abs{F_{1n}(s,X_s^{t,x})-F_{1m}(s,X_s^{t,x})}^qds\Big]\}^{\frac{1}{q}}\\
&\leq Ck^{-\frac{q-1}{q}}.
\end{aligned}
\end{equation*}
Therefore, for each $(t,x)\in [0,T]\times \R^m$, $(\varsigma^{1n}(t,x))_{n\geq 1}$ is a Cauthy sequence and then there
exists a borelian application $\varsigma^1$ on $[0,T]\times \R^m$, such that for each $(t,x)\in [0,T]\times \R^m$, $\lim_{n\rightarrow \infty}\varsigma^{1n}(t,x)=\varsigma^1(t,x)$, which indicates that for $t\in [0,T]$, $\lim_{n\rightarrow \infty} Y_t^{1n;(0,x)}(\omega)=\varsigma^1(t,\xto),\quad \bp-a.s.$ Taking account of \eqref{eq:estY1n} and Lebesgue dominated convergence theorem, we obtain the sequence $((Y_t^{1n;(0,x)})_{0\leq t\leq T})_{n\geq 1}$ converges to $Y^1=(\varsigma^1(t, \xto))_{0\leq t\leq T}$ in $\mathcal{L}^p([0,T]\times\R^m)$ for any $p\geq 1$, that is:
\begin{equation}\label{conv ytn}
\textbf{E}\Big[\int_{0}^T|Y_t^{1n;(0,x)}-Y_t^1|^p dt\Big]\rightarrow 0, \quad \text{as}\ n\rightarrow \infty.
\end{equation}
Next, we will show that for any $p>1$, $Z^{1n;(0,x)}=(\mathfrak{z}^{1n}(t,\xto))_{0\leq t\leq T})_{n\geq 1}$ has a limit in $\mathcal{H}_{T}^2(\R^m)$. Besides, $(Y^{1n;(0,x)})_{n\geq 1}$ is convergent in $\mathcal{S}_{T}^2(\R)$ as well.
We now focus on the first claim. For $n,m \geq 1$ and $0 \leq t\leq T$, using It\^{o}'s formula with $(Y_t^{1n}-Y_t^{1m})^2$ (we omit the subscript $(0,x)$ for convenience) and considering \eqref{abcd}-(b), we get,
\begin{align*}
&\big|Y_t^{1n}-Y_t^{1m}\big|^2+ \int_t^T |Z_s^{1n}-Z_s^{1m}|^2ds\\
&= 2\int_t^T(Y_s^{1n}-Y_s^{1m})(H_{1n}(s,
X_s^{0,x},Y_s^{1n},Y_s^{2n},Z_s^{1n},Z_s^{2n})-\\
&-H_{1m}(s, X_s^{0,x},Y_s^{1m},Y_s^{2m},Z_s^{1m},Z_s^{2m}))ds-2\int_t^T(Y_s^{1n}-Y_s^{1m})(Z_s^{1n}-Z_s^{1m})dB_s\\
&\leq C\int_t^T
\abs{Y_s^{1n}-Y_s^{1m}}\Big[(\abs{Z_s^{1n}}+\abs{Z_s^{1m}})(1+|X_s^{0,x}|)\\
&\qquad +|Y_s^{1n}|+|Y_s^{1m}|+ (1+|X_s^{0,x}|)^{\gamma}\Big]ds- 2\int_t^T(Y_s^{1n}-Y_s^{1m})(Z_s^{1n}-Z_s^{1m})dB_s.
\end{align*}
\no Since for any $x,y,z\in \R$, $|xyz|\leq \frac{1}{a}|x|^a+\frac{1}{b}|y|^b+\frac{1}{c}|z|^c$ with $\frac{1}{a}+\frac{1}{b}+\frac{1}{c}=1$, then, for any $\varepsilon>0$, we have,

\begin{equation}\label{4.27}
\begin{aligned}
&|Y_t^{1n}-Y_t^{1m}|^2+\int_t^T|Z_s^{1n}-Z_s^{1m}|^2ds\\
&\leq C\Big\{\frac{\varepsilon^2}{2}\int_t^T(|Z_s^{1n}|+|Z_s^{1m}|)^2ds+\frac{\varepsilon^4}{4}\int_t^T(1+|X_s^{0,x}|)^4ds\\
&\qquad \quad+\frac{1}{4\varepsilon^8}\int_t^T|Y_s^{1n}-Y_s^{1m}|^4ds+\frac{\varepsilon}{2}\int_t^T(|Y_s^{1n}|+|Y_s^{1m}|)^2ds\\
&\qquad \quad+\frac{\varepsilon}{2}\int_t^T(1+|X_s^{0,x}|)^{2\gamma} ds+\frac{1}{2\varepsilon}\int_t^T|Y_s^{1n}-Y_s^{1m}|^2ds\Big\}\\
&\quad-2\int_t^T(Y_s^{1n}-Y_s^{1m})(Z_s^{1n}-Z_s^{1m})dB_s.
\end{aligned}
\end{equation}
\no Taking now $t=0$ in \eqref{4.27}, expectation on both sides and the limit w.r.t. $n$ and $m$, we deduce that, 
\begin{equation}\label{4.29}
\limsup_{n,m\rightarrow \infty} \textbf{E}\Big[\int_{0}^T|Z_s^{1n}-Z_s^{1m}|^2ds\Big]\leq C\{\frac{\varepsilon^2}{2}+\frac{\varepsilon^4}{4}+\frac{\varepsilon}{2}\},
\end{equation}
due to \eqref{eq:estZ1n}, \eqref{eq:est_x} and the convergence of \eqref{conv ytn}. As $\varepsilon$ is arbitrary, then the sequence $(Z^{1n})_{n\geq 1}$ is convergent in $\mathcal{H}_{T}^2$ to a process $Z^1$.

Now, returning to inequality \eqref{4.27}, taking the supremum over $[0,T]$ and using BDG's inequality, we obtain that, 
\begin{align*}
&\textbf{E}\Big[\sup_{0\leq t\leq T}|Y_t^{1n}-Y_t^{1m}|^2+\int_{0}^T|Z_s^{1n}-Z_s^{1m}|^2ds\Big]\\
&\leq C\{\frac{\varepsilon^2}{2}+\frac{\varepsilon^4}{4}+\frac{\varepsilon}{2}\}+\frac{1}{4}\textbf{E}\Big[\sup_{0\leq t\leq T}|Y_t^{1n}-Y_t^{1m}|^2\Big]+4\textbf{E}\Big[\int_{0}^T|Z_s^{1n}-Z_s^{1m}|^2ds\Big]
\end{align*}
which implies that
$$
\limsup_{n,m\rightarrow \infty}\textbf{E}\Big[\sup_{0\leq t\leq T}|Y_t^{1n}-Y_t^{1m}|^2\Big]=0,
$$
since $\varepsilon$ is arbitrary and \eqref{4.29}. Thus, the sequence of $(Y^{1n})_{n\geq 1}$ converges to $Y^1$ in $\mathcal{S}_{T}^2$ which is a continuous process.

Finally, repeat the procedure for player $i=2$, we have also the convergence of $(Z^{2n})_{n\geq 1}$ (resp. $(Y^{2n})_{n\geq 1}$) in $\mathcal{H}_{T}^2$ (resp. $\mathcal{S}_{T}^2$) to $Z^2$ (resp. $Y^2=\varsigma^2(.,  X^{0,x}.)$). 

\paragraph{Step 4.} The limit process $(Y_t^i,Z_t^i)_{0\leq t\leq T}, i=1,2$ is the solution of BSDE \eqref{eq:mainbsdeinth}.

Indeed, we need to show that (for case $i=1$):
\[
F_1(t,\xto)=H_1(s, \xso, Y_s^1,Y_s^2,Z_s^1,Z_s^2)\quad dt\otimes d\bp- a.s.
\]
\no For $k\geq 1$, we have, 
\begin{align} \label{G}
&\textbf{E}\bigg[\int_{0}^T |H_{1n}(s, \xso, Y_s^{1n},Y_s^{2n},Z_s^{1n},Z_s^{2n})- H_1(s, \xso, Y_s^1,Y_s^2,Z_s^1,Z_s^2)|ds\bigg]\nonumber\\
&\leq \textbf{E}\bigg[\int_{0}^T |H_{1n}(s, \xso, Y_s^{1n},Y_s^{2n},Z_s^{1n},Z_s^{2n}) \nonumber\\
&\qquad - H_1(s, \xso, Y_s^{1n},Y_s^{2n},Z_s^{1n},Z_s^{2n})|\cdot 1_{\{ |Y_s^{1n}|+|Y_s^{2n}|+|Z_s^{1n}|+|Z_s^{2n}|<k\}}ds\bigg]\nonumber\\
&\quad +\textbf{E}\bigg[\int_{0}^T |H_{1n}(s, \xso, Y_s^{1n},Y_s^{2n},Z_s^{1n},Z_s^{2n}) \nonumber\\
&\qquad- H_1(s, \xso, Y_s^{1n},Y_s^{2n},Z_s^{1n},Z_s^{2n})| \cdot 1_{\{ |Y_s^{1n}|+|Y_s^{2n}|+|Z_s^{1n}|+|Z_s^{2n}|\geq k\}}ds\bigg]\nonumber\\
&\quad +\textbf{E}\bigg[\int_{0}^T |H_1(s, \xso,Y_s^{1n},Y_s^{2n},Z_s^{1n},Z_s^{2n})- H_1(s, \xso, Y_s^1,Y_s^2,Z_s^1,Z_s^2)|ds\bigg]\nonumber\\
&:= I_1^n+I_2^n+I_3^n.
\end{align}
\no The sequence $I_1^n$, $n\geq 1$ converges to 0. On one hand, for $n\geq 1$, the point \eqref{abcd}-(b) in Step 1 implies that,
\begin{align*}
&|H_{1n}(s, \xso, Y_s^{1n},Y_s^{2n},Z_s^{1n},Z_s^{2n}) - H_1(s, \xso,  Y_s^{1n},Y_s^{2n},Z_s^{1n},Z_s^{2n})|\\
&\cdot 1_{\{ |Y_s^{1n}|+|Y_s^{2n}|+|Z_s^{1n}|+|Z_s^{2n}|<k\}}\nonumber\\
& \leq C_2(1+|\xso|)k+C_h(1+|\xso|^{\gamma}+k).
\end{align*}
\no On the other hand, considering the point \eqref{abcd}-(d), we obtain that, 
\begin{align*}
&|H_{1n}(s, \xso, Y_s^{1n},Y_s^{2n},Z_s^{1n},Z_s^{2n}) - H_1(s, \xso,  Y_s^{1n},Y_s^{2n},Z_s^{1n},Z_s^{2n})|\\
&\cdot 1_{\{ |Y_s^{1n}|+|Y_s^{2n}|+|Z_s^{1n}|+|Z_s^{2n}|<k\}}\nonumber\\
&\leq\quad \sup\limits_{\substack{(y_s^1,y_s^2,z_s^1,z_s^2)\\|y_s^1|+|y_s^2|+|z_s^1|+|z_s^1|<k}}|H_{1n}(s, \xso, y_s^{1},y_s^{2},z_s^{1},z_s^{2})- H_1(s, \xso, y_s^{1},y_s^{2},z_s^{1},z_s^{2})|\\
&\rightarrow 0,
\end{align*} 
as $n\rightarrow \infty$. Thanks to Lebesgue's dominated convergence theorem, the sequence $I_1^n$ of (\ref{G}) converges to 0 in $\mathcal{H}_{T}^1$.

\no The sequence $I_2^n$ in \eqref{G} is bounded by $\frac{C}{k^{2(q-1)/q}}$ with $q\in (1,2)$. Actually, from point \eqref{abcd}-(b) and Markov's inequality, for $q\in (1,2)$, we get,
\begin{align*}
I_2^n &\leq C\Big\{\textbf{E}\Big[\int_{0}^T(1+|\xso|)^q|Z_s^{1n}|^q+(1+|\xso|^{\gamma}+|Y_s^{1n}|)^qds\Big]\Big\}^{\frac{1}{q}}\times\\
&\qquad\qquad\qquad\times \Big\{\textbf{E}\Big[\int_{0}^T 1_{\{|Y_s^{1n}|+|Y_s^{2n}|+|Z_s^{1n}|+|Z_s^{2n}|\geq k\}}ds\Big]\Big\}^{\frac{q-1}{q}}\\
&\leq  C\Big\{\textbf{E}\Big[\int_{0}^T|Z_s^{1n}|^2ds\Big]+\textbf{E}\Big[\int_{0}^T (1+|\xso|)^{\frac{2q}{2-q}}ds\Big]\\
&\qquad\qquad\qquad+\textbf{E}\Big[\int_{0}^T|Y_s^{1n}|^2ds\Big]+\textbf{E}\Big[\int_{0}^T(1+|\xso|)^{\gamma q}ds\Big]\Big\}^{\frac{1}{q}}\times \\
&\quad\times \frac{\Big\{\textbf{E}\Big[\int_{0}^T|Y_s^{1n}|^2+|Y_s^{2n}|^2+|Z_s^{1n}|^2+|Z_s^{2n}|^2ds\Big]\Big\}^{\frac{q-1}{q}}}{(k^2)^{\frac{q-1}{q}}}\\
&\leq \frac{C}{k^{\frac{2(q-1)}{q}}}.
\end{align*}
\no The last inequality is a straightforward result of the estimates (\ref{eq:est_x})(\ref{eq:estY1n}) and (\ref{eq:estZ1n}). 
\medskip

The third sequence $I_3^n$, $n\geq 1$ in \eqref{G} also converges to $0$, at least for a subsequence. Actually, since the sequence $(Z^{1n})_{n\geq 1}$ converges to $Z^1$ in $\mathcal{H}_{T}^2$, then there exists a subsequence $(Z^{1n_k})_{k\geq 1}$ such that it converges to $Z^1$, $dt\otimes d\bp$-a.e., and furthermore, $\sup_{k\geq 1}|Z_t^{1n_k}(\omega)|\in \mathcal{H}_{T}^2$. On the other hand, $(Y^{1n_k})_{k\geq 1}$ converges to $Y^1>0$, $dt\otimes d\bp$-a.e.. Thus, by the continuity of function $H_1(t,x,y^1,y^2,z^1,z^2)$ w.r.t $(y^1,y^2,z^1,z^2)$, we obtain that
\begin{align*}
H_1(t,\xto,&Y_t^{1n_k},Y_t^{2n_k},Z_t^{1n_k},Z_t^{2n_k})\\
&\xrightarrow{k\rightarrow\infty} H_1(t,\xto,Y_t^{1},Y_t^{2},Z_t^{1},Z_t^{2})\quad \ dt\otimes d\bp-a.e.
\end{align*}
In addition, considering that
\[
\sup_{k\geq 1}|H_1(t,\xto,Y_t^{1n_k},Y_t^{2n_k},Z_t^{1n_k},Z_t^{2n_k})| \in \mathcal{H}_{T}^q(\R)\text{ for } 1<q<2,
\]
which follows from (\ref{eq:F1nlp}). Finally, by the dominated convergence theorem, one can get that,
\begin{align*}
H_1(t,\xto,&Y_t^{1n_k},Y_t^{2n_k},Z_t^{1n_k},Z_t^{2n_k})\\
&\xrightarrow{k\rightarrow\infty} H_1(t,\xto,Y_t^{1},Y_t^{2},Z_t^{1},Z_t^{2}) \text{ in } \mathcal{H}_T^q,
\end{align*}
which yields to the convergence of $I_3^n$ in (\ref{G}) to $0$.

It follows that the sequence $(H_{1n}(t,\xto,Y_t^{1n},Y_t^{2n},Z_t^{1n},Z_t^{2n})_{0\leq t\leq T})_{n\geq 1}$ converges to 
 $(H_1(t,\xto,Y_t^{1},Y_t^{2},Z_t^{1},Z_t^{2}))_{0\leq t\leq T}$ in $\mathcal{L}^1([0,T]\times \Omega, dt\otimes d\bp)$ and then $F_1(t,\xto)=H_1(t,\xto,Y_t^{1},Y_t^{2},Z_t^{1},Z_t^{2}), dt\otimes d\bp$-a.e. In the same way, we have, $F_2(t, X_t^{0,x})=H_2(t,\xto,Y_t^{1},Y_t^{2},Z_t^{1},Z_t^{2})$, $dt\otimes d\bp$-a.e. Thus, the processes $(Y^i, Z^i)$, $i=1,2$ are the solutions of the backward equation (\ref{eq:mainbsdeinth}).

\end{proof}


\begin{backmatter}

\end{backmatter}
\end{document}